\documentclass[12pt,a4paper]{article}

\usepackage[top=3cm, bottom=3.5cm, left=2cm, right=2cm]{geometry}
\usepackage[T1]{fontenc}
\usepackage[]{geometry}
\usepackage{graphicx}
\usepackage{mathtools}
\usepackage{amssymb}
\usepackage{amsthm}
\usepackage{thmtools}
\usepackage{xcolor}
\usepackage{nameref}
\usepackage{hyperref}
\usepackage{enumitem}
\usepackage{tikz}
\usepackage{ytableau}
\usetikzlibrary{decorations.pathreplacing,calc}

\usepackage[utf8]{inputenc}

\usepackage{tabularx}
\usepackage[numbers]{natbib}
\usepackage{pgfplots}
\usetikzlibrary{arrows.meta}
\usepackage{framed}
\usepackage{thmtools}
\usepackage{longtable}
\usepackage{fancybox}
\usepackage{float}
\usepackage{pifont}
\usepackage{pdfpages}
\usepackage{calc}
\usepackage{varioref}
\usepackage{pgf}
\usepackage{graphicx}
\usepackage{varioref}
\usepackage{caption}
\usepackage{titling}
\usepackage{mathrsfs}
\usetikzlibrary{arrows}
\usetikzlibrary{plotmarks}
\usetikzlibrary{positioning}
\usetikzlibrary{automata}
\usetikzlibrary{decorations.markings}
\usetikzlibrary{shapes.arrows}
\usepackage{amsthm}
\usepackage{fancyhdr}
\usepackage[explicit]{titlesec}
\usepackage{lmodern}
\usepackage{lipsum}
\usepackage{color}
\newtheorem{theorem}{Theorem}[section]
\newtheorem{corollary}[theorem]{Corollary}
\newtheorem{lemma}[theorem]{Lemma}
\newtheorem{example}[theorem]{Example}
\theoremstyle{definition} 
\newtheorem{remark}[theorem]{Remark}
\theoremstyle{remark}

\theoremstyle{examples}

\newtheorem{appli}[theorem]{Applications}
\usepackage{lmodern}
\usepackage{array}
\usepackage{caption}
\usepackage[none]{hyphenat}
\usepackage{float}

\newtheorem*{theoremts}{Theorem}

\newenvironment{porofo}[1][Proof]{%
	\begin{proof}[#1]%
	}{%
	\end{proof}%
}

\newcommand{\cheg}{\mathcal{R}_{n,k}}

\newcommand{\chegg}{\mathcal{R}_{n,k}}

\newcommand{\sui}[3]{\mathrm{#1}(#2,#3)}

\newcommand{\suit}[1]{{#1}(n,k)}

\newcommand{\coul}[1]{#1_{n,k}}
\newcommand{\cou}[3]{#1_{#2,#3}}
\newcommand{\poi}[1]{#1_{i,j}}

\newcommand{\idc}{\in \natu}
\newcommand{\argtgen}[3]{(#1|#2)^{(\overline{#3})}}

\newcommand{\argtgenm}[3]{(#1|#2)^{(\underline{#3})}}

\newcommand{\es}{\textbf{E}}
\newcommand{\nes}{\textbf{NE}}

\newcommand{\ma}[1]{\mathbb{#1}}

\newcommand{\st}[2]{\mathcal{#1}_{#2}}

\newcommand{\fiel}[1]{\mathbb{#1}}
\newcommand{\real}{\mathbb{R}}
\newcommand{\natu}{\mathbb{N}}

\newcommand{\coss}{\cdots}

\newcommand{\reel}{\mathbb{R}}

\newcommand{\poiw}[1]{\mathrm{#1}}

\newcommand{\nkid}{}

\begin{document}

\begin{center}
	
	{\huge\bf Explicit Formulas and Combinatorial Interpretation of Triangular Arrays via Weighted Paths}
	
	\vspace{0.5cm}
	
	\textbf{AUBERT Voalaza Mahavily Romuald} \\
	\texttt{aubert@aims.ac.za} \\
	Department of Mathematics, Faculty of Science \\
	Laboratoire de Mathématiques et Applications de l’Université de Fianarantsoa (LaMAF) \\
	University of Fianarantsoa \\
	Postal Adress 601, Fianarantsoa, Madagascar
	
	\vspace{0.5cm}
	
	\textbf{RANDRIANIRINA Benjamin} \\
	\texttt{benjamin.randrianirina@univ-fianarantsoa.mg} \\
	Department of Mathematics, Faculty of Science \\
	Laboratoire de Mathématiques et Applications de l’Université de Fianarantsoa (LaMAF) \\
	University of Fianarantsoa \\
	Postal Adres 601, Fianarantsoa, Madagascar
\end{center}
	
	\section*{Abstract} 
	Using the lattice paths in  $\mathbb{N}\times\mathbb{N}$, we derive a general formula for sequences $\big(\suit{T}\big)\nkid$ satisfying the recurrence relation of the form:
	\begin{equation*}
	\suit{T}=\coul{a}T(n-1,k)+\coul{b}T(n-1,k-1).
	\end{equation*}
	We apply this result to the  case where $a_{n,k}=a_0+a_1k+a_2n$ and $b_{n,k}=b_0+b_1k+b_2n$. 
	This leads to explicit expressions for $\suit{T}$, with simpler formulas arising in the case $b_2=0$, as well as in the fully general case, using Fa\`a di Bruno's type expression. In particular, we analyze the case  $b_{n,k}=1$, which frequently occurs in enumerative combinatorics. Applications include explicit formulas for the $r$-Eulerian numbers.We also express the case $b_{n,k}=1$, using a transition matrix. We apply our results to several sequences.\\
	\textbf{Keywords:} triangular recurrence, weighted paths, $r$-Eulerian numbers, combinatorial interpretation.
	\\ \textbf{Mathematics Subject Classification}: 05A05; 05A15; 05A19

	\section{Introduction }
	\noindent	In this paper, we investigate certain sequence of numbers that satisfy a triangular recurrence of the form 
		\begin{equation}\label{recu}
		\suit{T}=(a_0+a_1k+a_2n)T(n-1,k)+(b_0+b_1k+b_2n)T(n-1,k-1).
		\end{equation}
		In their book \cite{GrahamKnuthPatashnik1994}, Graham et al. study binomial coefficients, Stirling numbers, and Eulerian numbers, and propose a generalization problem of the form \eqref{recu}. We call the resulting sequences GKP numbers. In combinatorics, several approaches have been developed to study these numbers. One of the earliest results is due to Neuwirth~\cite{neuwirth2001} (see also Spivey~\cite{Spivey2011JIS}), who obtained an explicit formula for the case $b_2=0$ using the Galton triangle. Spivey~\cite{Spivey2011JIS} further investigated several cases using finite differences. Analytical approaches have also appeared in various works, including those of Théorêt~\cite{theoret1994hyperbinomiales}, Wilf~\cite{wilf2004}, and Barbero~\cite{BARBEROG2014146}.

		Recall that a weighted set is a pair $(E,v)$, where $E$ is a finite nonempty set and $v$ is a map from $E$ to a ring of formal power series with coefficients in $\mathbb{C}$. We are generally interested in $|E|_v=\sum_{x\in E}v(x)$. 
		Using weighted paths in the  $\mathbb{N}\times\mathbb{N}$, we derive a general formula for the sequence $\big(\suit{T}\big)\nkid$ of the form:
		\begin{equation}
		\label{recu2}
		\suit{T}=\coul{a}T(n-1,k)+\coul{b}T(n-1,k-1).
		\end{equation}
 We apply this result to the special case of GKP numbers.
 Thus, the aim of this work is to study such numbers as total weights of certain paths, with initial conditions usually taken as
	\begin{equation}\label{condiusue}
	T(0,0)=1, \quad \suit{T}=0 \text{ if } n<\max(k,0).
	\end{equation}
	We remain in the setting where $a_{n,k}=a_0+a_1k+a_2n$ and $b_{n,k}=b_0+b_1k+b_2n$, with $a_{n,k}\neq 0$ and $b_{n,k}\neq 0$. In particular, we study the case  $\coul{b}=1$, which we denote mostly by $\big(F(n,k)\big)\nkid$.  In this way, we obtain a unified approach to generalized Stirling numbers (see Hsu and Shiue \cite{hsu1998unified}, Maier \cite{Maier2023}), the $r$-Whitney numbers (see \cite{Mezo2010Bernoulli}),  the $r$-Lah numbers (see Ni\'ul and R\'acz \cite{Niul-Racz}) and many others sequences in combinatorics.

	 However, generalizations of Eulerian numbers where $\coul{b}\neq 1$ are also important examples when discussing sequences satisfying \eqref{recu}. We provide an explicit formula for $T(n,k)$ with $a_{n,k}=a_0+a_1k+a_2n$ and $b_{n,k}=b_0+b_1k$, and apply it on the $r$-Eulerian numbers.

	\section*{Our Contributions and Paper Outline}
\noindent	We prove the main results, which provide explicit formulas for $T(n,k)$ of Fa\`a di Bruno's type expressions:
\begin{theoremts}[\ref{theocontr}]
	Let $\big(\suit{T}\big)$ be a sequence satisfying \eqref{recu} with $\coul{a} = a_1 k + a_0$, $\coul{b} = b_2n+b_1k+b_0$, and the usual initial conditions \eqref{condiusue}. 
	Then
	\begin{equation*}
		\suit{T} = \sum_{\{c_0+c_1+\coss+c_{k}=n-k, 0\leq c_i\leq n\}}  a_0^{c_0}\prod_{i=1}^{k} \big(a_0+a_1i\big)^{c_{i}} \Big(b_0+ (b_1+b_2)i+b_2 (\sum_{j=0}^{j=i-1}c_j) \Big).
	\end{equation*}
\end{theoremts}
\noindent More generally, we also obtain the following theorem.
\begin{theoremts}[\ref{theocontr2}]
	Let $\big(\suit{T}\big)$ be a sequence satisfying \eqref{recu} and the usual initial conditions \eqref{condiusue}. 
	Then
	\begin{equation*}
		\suit{T}
		= \sum_{\{1\leq p_1<p_2<\cdots <p_{n-k}\leq n\}}
		\prod_{i=1}^{n-k} \Big( (a_2+a_1)p_i-a_1 i+a_0 \Big)
		\prod_{i=1}^{k} \Big(b_0+ (b_1+b_2)i+b_2 \sum_{j=0}^{i-1}\# \{l: p_l-l=j\} \Big).
	\end{equation*}
\end{theoremts}
In section \ref{sect2}, we begin to interpret $T(n,k)$ as  a total weight of a set of paths and derive explicit formula. It includes theorems \ref{princitheo}, corollary \ref{corrnerw} and the two above theorems. We also apply these results to sequences defined by triangular recurrence. 

Section \ref{sect3} contains brief remarks that we can also represent the sequences as transition matrices. We prove that the case where $b_{n,k}=1,$ it is a transition matrix in the vector space $k[x]$.

In section \ref{sect4}, we present a simpler formula for the case $b_2=0$.
	\section{Weighted paths in $\mathbb{N}\times \mathbb{N}$ associated with sequences of numbers satisfy a triangular recurrence}\label{sect2}
		We define $\chegg$ as the set of paths starting from $(0,0)$ to $(n,k)$ in the lattice $\natu \times \natu$, using only 
		East steps \es\ \big($(i-1,j)\to (i,j)$\big) and North-East steps \nes\ \big($(i-1,j-1)\to (i,j)$\big).  Any element $\pi\in\mathcal{R}_{n,k}$ is a path of length $n$ and height $k$.
				
			Let $a=(\cou{a}{i}{j})_{i,j\in \natu}$ and $b=(\cou{b}{i}{j})_{i,j\in \natu}$ be two sequences of real numbers.  We assign a weight to each East step 
		\es\ $(i-1,j)\to (i,j)$ by $\cou{a}{i}{j}$, and to each North-East step  
		\nes\ $(i-1,j-1)\to (i,j)$ by $\cou{b}{i}{j}$ so one can define  a valuation (weighting)  
		$\poiw{\omega}_{a,b}:{\chegg}\to{\reel}$, such that the weight of a path $\pi=(p_1,p_2,\cdots,p_{n+k})\in \chegg$ is
		\(\poiw{\omega}_{a,b}(\pi):=\) the product of the weights of each $p_i$.
			We denote by $\suit{T_{a,b}} = |\chegg|_{\poiw{\omega}_{a,b}}= \sum_{\pi \in \cheg} \poiw{\omega_{a,b}}(\pi)$ the total weight of $\chegg$, and by  
		$\ma{T}_{a,b} = (\suit{T_{a,b}})_{n,k\in \natu}$ the associated infinite matrix.
	\begin{lemma}\label{rem2}
		The matrix 	$\ma{T}_{a,b}$ satisfies the following recurrence relation.
		\begin{equation}\label{relprinc} 
		\suit{T_{a,b}}=\coul{a}\sui{T_{a,b}}{n-1}{k}
		+\coul{b}\sui{T_{a,b}}{n-1}{k-1}.
		\end{equation}
		It satisfies also the usual initial condition \eqref{condiusue}. 
		Conversely, any relation of this form, for $\coul{a},\coul{b}\in \real$, may be interpreted as a weighting of $\chegg$.
	\end{lemma}
	\begin{porofo} Indeed, any path from $(0,0)$ to $(n,k)$ is either a path from $(0,0)$ to $(n-1,k)$ followed by an East step, or a path from $(0,0)$ to $(n-1,k-1)$ followed by an East step.
		\end{porofo}

\begin{example}\rm Figure \ref{saryresauex} represents a path from $(0,0)$ to $(6,3)$ whose weight is 
$a_{1,0}b_{2,1}b_{3,2}a_{4,2}b_{5,3}a_{6,3}$.
\end{example}
	\begin{figure}[H]
		\centering
		\begin{tikzpicture}[scale=0.9]
		\begin{axis}[
		axis lines=middle,
		axis line style={-Stealth},
		width=0.8\linewidth,
		enlargelimits=false,
		xmin=0, xmax=9,
		ymin=0, ymax=9,
		grid=both,
		xtick={0,...,7},
		ytick={0,...,7},
		xlabel={$n$},
		ylabel={$k$},
		clip=false,
		axis on top=true,
		]
		
		\addplot[color=red, domain=0:9, samples=2] {x};
		
		\pgfmathsetmacro{\yzero}{(0 - (-1))/(9 - (-1))}
		\pgfmathsetmacro{\xzero}{(0 - 0)/(9 - 0)}
		\node[anchor=west]  at (rel axis cs:1,\yzero) {East};
		\node[anchor=south] at (rel axis cs:\xzero,1) {North};

		\addplot[very thick, draw=blue] coordinates {(0,0) (1,0)};
		\addplot[very thick, draw=blue] coordinates {(3,2) (4,2)};
		\addplot[very thick, draw=blue] coordinates {(5,3) (6,3)};
		
		\addplot[very thick, draw=gray] coordinates {(1,0) (2,1)};
		\addplot[very thick, draw=gray] coordinates {(2,1) (3,2)};
		\addplot[very thick, draw=gray] coordinates {(4,2) (5,3)};
		
		\addplot[only marks, mark=*, mark size=2.2pt]
		coordinates {(0,0) (1,0) (2,1) (3,2) (4,2) (5,3) (6,3)};
		
		\end{axis}
		\end{tikzpicture}
		\caption{A path from $(0,0)$ to $(6,3)$}
		\label{saryresauex}
		\end{figure}
	
		Let $\mathcal{C}^\uparrow(k,n)$ be the set of strictly increasing functions from $[k]$ to $[n]$.
	We identify $\pi \in \cheg$ with $\sigma \in \mathcal{C}^\uparrow(k,n)$ by defining $\sigma(i)=j$ if the $i$-th \nes\ step is $(j-1,l-1)\to (j,l)$. 
	Given $\sigma \in \mathcal{C}^\uparrow(k,n)$, one can associate the increasing list $\tilde{\sigma}=(\tilde{\sigma}_i)_{1\leq i\leq n-k}$ of the elements of $[n]\setminus \sigma([k])$, and $|\mathcal{C}^\uparrow(k,n)|={n\choose k}$.
	\begin{example}\rm The path in Figure \ref{saryresauex} is identified with the map $\sigma$ such that $\sigma(1)=2$, $\sigma(2)=3$ and $\sigma(3)=5$, and $\tilde{\sigma}=(1,4,6)$.  
	\end{example}
	\begin{theorem}\label{princitheo}
		Let $a=(a_{n,k})_{n,k\in \natu}$, $b=(b_{n,k})_{n,k\in \natu}$ and $T=\big(\suit{T}\big)\nkid$ be three sequences such that $T$ satisfies the recurrence relation
	$\displaystyle	\suit{T}=\coul{a}T(n-1,k)+\coul{b}T(n-1,k-1),$
		with the initial conditions \eqref{condiusue}. Then
		\begin{equation}\label{pprinceq}
		\suit{T}
		=\sum_{\sigma \in \mathcal{C}^\uparrow(k,n)}
		\prod_{i=1}^{n-k} \big( \cou{a}{\tilde{\sigma}_i}{\tilde{\sigma}_i-i} \big)
		\prod_{i=1}^{k} \big( \cou{b}{\sigma(i)}{i} \big).
		\end{equation}
	\end{theorem}
	
	\begin{porofo}
		Recall our notation: $\omega_{a,b} : \cheg \to \mathbb{R}$, the weighting, assigns
		the weight $\poi{a}$ to each East step \es\ $(i-1,j) \to (i,j)$ and the weight
		$\poi{b}$ to each North-East step \nes\ $(i-1,j-1) \to (i,j)$. Since $\suit{T}$
		is the total weight of $\cheg$, we have
		$\displaystyle
		\suit{T} = \sum_{\pi \in \cheg} \poiw{\omega_{a,b}}(\pi).
		$
		
		It is straightforward to see that each East step is weighted by $a_{\tilde{\sigma_i},\tilde{\sigma_i}-i}$ and each North-East step is weighted by $b_{\sigma(i),i}$. Thus, the result follows.
	\end{porofo}

The case $b_{n,k}=1$, corresponding to sequences $(F(n,k))$ satisfying the recurrence relation:
\begin{equation}\label{bbbcas}
	F(n,k) = F(n-1,k-1) + (a_2n+a_1k+a_0) F(n-1,k),
\end{equation}
with the usual initial conditions \eqref{condiusue}, is very common in combinatorics. In this case, \eqref{pprinceq} becomes 
	\begin{equation}\label{eqqq11}
	\suit{F}
	=\sum_{\sigma \in \mathcal{C}^\uparrow(k,n)}
	\prod_{i=1}^{n-k} \big( \cou{a}{\tilde{\sigma}_i}{\tilde{\sigma}_i-i} \big).
\end{equation}
Since $a_{n,k}=a_2n+a_1k+a_0$, we obtain 
\begin{corollary}\label{cor11} The sequence $(F(n,k))$ satisfies the following:
	\begin{equation}\label{formprinc}
		\suit{F}= \sum_{\{1\leq p_1<p_2<\coss <p_{n-k}\leq n\}}  
		\prod_{i=1}^{n-k} \Big( (a_2+a_1)p_i-a_1i+a_0 \Big).
	\end{equation} 
\end{corollary}

	\begin{lemma}\label{cor1} If $b_{n,k}=b_0+b_1k$, then 
		\begin{equation}\label{pprinceq1}
		\suit{T}
		=\argtgen{b_0 + b_1}{\, b_1 \,}{k}F(n,k),
		\end{equation} where  $\argtgen{b_0 + b_1}{\, b_1 \,}{k}= (b_0+b_1)(b_0+2b_1)\cdots(b_0+kb_1)$.\\
		In particular, if $b_1=0$ then: 
		\begin{equation}\label{pprinceq2}
		\suit{T}=b_0^k F(n,k).
		\end{equation}
		And if $b_0=0$, then:
		\begin{equation}\label{pprinceq3}
		\suit{T}=b_1^k k! F(n,k).
		\end{equation} 
		\end{lemma}
\begin{proof}
	Suppose that $b_{n,k} = b_0 + b_1 k$. Since each element $\pi \in \mathcal{R}_{n,k}$ has $k$ \nes\ steps, and the weight of each of these steps is independent of $n$, the step at height $1$ is weighted by $b_0 + b_1$, the step at height $2$ by $b_0 + 2b_1$, $\cdots$, and the step at height $k$ by $b_0 + k b_1$. Hence, the total weight of the \nes\ steps is
	\[
	(b_0 + b_1)(b_0 + 2b_1)\cdots(b_0 + k b_1).
	\]
	
	\noindent This means that, for the same $a_{n,k}$, if the weight of $\pi \in \cheg$ is $\omega(\pi)$ in the case $b_{n,k} = 1$, and $\omega'(\pi)$ in the case $b_{n,k} = b_0 + b_1 k$, then
	\[
	\omega'(\pi) = (b_0 + b_1)(b_0 + 2b_1)\cdots(b_0 + k b_1)\,\omega(\pi).
	\] Thus we can conclude \eqref{pprinceq1}.
\end{proof}

\noindent Combining Corollary \ref{cor11} and Lemma \ref{cor1}, we obtain
\begin{corollary}\label{corrnerw}
	Let $\big(\suit{T}\big)$ be a sequence satisfying \eqref{recu} with $\coul{a} = a_2n+a_1 k + a_0$, $\coul{b} = b_1k+b_0$, and the usual initial conditions \eqref{condiusue}. 
	Then
	\begin{equation}
		\suit{T} =\argtgen{b_0 + b_1}{\, b_1 \,}{k} \sum_{\{1\leq p_1<p_2<\coss <p_{n-k}\leq n\}}  
		\prod_{i=1}^{n-k} \Big( (a_2+a_1)p_i-a_1i+a_0 \Big).
	\end{equation}
\end{corollary}

\begin{appli}\rm\label{appli1}
	\begin{enumerate}
\item Setting $a_2=-\alpha$, $a_1=\beta$, $a_0=\alpha+r$, we obtain the generalized Stirling Numbers $S_{\alpha,\beta,r}$ defined by Hsu and Shiue in \cite{hsu1998unified}, satisfying the recurrence relation (\cite{hsu1998unified} Equation (7)):
\begin{equation}
S_{\alpha,\beta,r}(n,k)=S_{\alpha,\beta,r}(n-1,k-1)+(-\alpha (n-1)+\beta k+r)S_{\alpha,\beta,r}(n-1,k).
\end{equation}
The generalized Stirling Numbers verify:
\begin{equation}\label{genestir}
S_{\alpha,\beta,r}(n,k)= \sum_{\{1\leq p_1<p_2<\coss <p_{n-k}\leq n\}}  
\prod_{i=1}^{n-k} \Big( (\beta-\alpha)p_i-\beta i+\alpha+r \Big).
\end{equation} 
Note that $S_{m,0,r}(n;k)$ is  the $r$-Whitney number of first kind, defined by Mez\"o in \cite{Mezo2010Bernoulli}, and  $S_{0,m,r}(n;k)$ is  the $r$-Whitney number of second kind.

 In \cite{rabeza}, Randrianirina introduces several generalizations of the Lah numbers.

\item \label{br1} If we consider the sequence defined by the recurrence
relation:
\begin{equation}
\lambda_{m,r}(n,k)=\lambda_{m,r}(n,k-1)+(m(n-1)+k+r)\lambda_{m,r}(n,k),
\label{equav3}
\end{equation}
we have $a_0=-m+r$; $a_1=1$ and $a_2=m$. Then:

\begin{equation}\label{equabase2}
\lambda_{m,r}(n,k)=\sum_{\{1\leq p_1<p_2<\coss <p_{n-k}\leq n\}}  
\prod_{i=1}^{n-k} \Big( (m+1)p_i-i-m+r \Big).
\end{equation}
$\lambda_{m,r}(n,k)$ is the matrix associated with the $\mathbb{L}$-species (see \cite{rabeza2}) $z\mathbb{E}_\beta\circ \mathcal{Q}_m^+)\cdot\mathcal{Q}_m^r=T_{Z=z,X=x}\mathcal{G}(ZX^r)$, $\mathcal{G}$ being the combinatorial differential operator associated with the grammar $G=\{z\to \mu zx^{m+1}, x\to x^{m+1}\}$, $\mathbb{E}_\beta$ is the species of sets weithed by $w(A)=\mu^{|A|},$ and $\mathcal{Q}_m$ is the $\mathbb{L}$-species of $m$-Stirling permutation, weighted by $w'(\sigma)=x^{mn+1}$ if $\sigma\in\mathcal{Q}[n]$.

\item	Setting $a_2=a_1=m$, $a_0=r+s$, we obtain the numbers $\Lambda_{m,r,s}(n,k)$, satisfying the relation:
\begin{equation}
\label{equak}
\Lambda_{m,r,s}(n,k)=\Lambda_{m,r,s}(n-1,k-1)+(m(n+k)+r+s)\Lambda_{m,r,s}(n-1,k).
\end{equation}
We have \begin{equation}
\Lambda_{m,r,s}(n,k)= \sum_{\{1\leq p_1<p_2<\coss <p_{n-k}\leq n\}}  
\prod_{i=1}^{n-k} \Big( m(2p_i-i)+r+s \Big).
\end{equation} 
$\Lambda_{m,r,s}(n,k)$ is the matrix associated with the $\mathbb{L}$-species $\mathbb{F}=T_{V=v,U=u,X=x}\mathcal{G}(VU^sX^r)$, where
$\mathcal{G}$ is the combinatorial differential operator associated with the grammar $G=\{v \to \mu vu^m x^m, u \rightarrow ux^m , x \to x^{m+1} \}$.

Note that $\Lambda_{m,r,r}(n,k)$ is the $r$-Whitney-Lah numbers defined by
Gyimesi and  Ny\'ul in \cite{Gyi-Nyul}, and $\Lambda_{1,r,r}(n,k)$ is the $r$-Lah numbers defined by Ny\'ul and R\'acz \cite{Niul-Racz})
\item If we set $a_1=a_2=m$ and $a_0=mr$, we obtain the sequence ($\tau_{m,r}(n,k)$)
satisfying the relation: 
\begin{equation} 
\label{equatau}
\tau_{m,r}(n,k)=\tau_{m,r}(n-1,k-1)+m(n+k+r)\tau_{m,r}(n-1,k),
\end{equation}
and associated with $(z\mathbb{E}_\beta \circ(\mathbb{L}_{\gamma})^+)\cdot(\mathbb{L}_{\gamma'})^r=T_{Z=z,X=x}\mathcal{G}(ZX^r)$, where
$\mathcal{G}$ is the combinatorial differential operator associated with the grammar $G=\{v \to \mu z x^r, x \to mx^{2} \}$.
We have 
\begin{equation}\label{equabase3}
\tau_{m,r}(n,k)= \sum_{\{1\leq p_1<p_2<\coss <p_{n-k}\leq n\}}  
\prod_{i=1}^{n-k} \Big( m(2p_i-i+r \Big).
\end{equation}
\end{enumerate}

\end{appli}

To continue, let's see an useful tool.
Let the sets $A,B,C$ be defined by
\begin{align*}
	A=\mathcal{C}^\uparrow(k,n)=&\{\sigma=(\sigma_1,\sigma_2,\coss,\sigma_{k}): 1\leq \sigma_1<\sigma_2<\coss <\sigma_{k}\leq n\}&;\\
	B=\mathcal{C}^\uparrow(n-k,n)=&\{\tilde{\sigma}=(\tilde{\sigma}_1,\tilde{\sigma}_2,\coss,\tilde{\sigma}_{n-k}): 1\leq \tilde{\sigma}_1<\tilde{\sigma}_2<\coss <\tilde{\sigma}_{n-k}\leq n\} \text{ and }\\
	C=&\{ \underline{c}=(c_0,c_1,\coss,c_{k}): \sum_{i=0}^{k}c_i=n-k, \quad \forall i, c_i\in \{0,1,\coss, n\} \}.
\end{align*}
We equip $B$ and $C$ with the weights $\beta$ and $\delta$ defined by
\begin{equation*}
	\beta(\tilde{\sigma}):=\prod_{j=1}^{n-k} \big(a_0+(\tilde{\sigma}_j-j)a_1\big)
	\text{ and } 
	\delta(\underline{c}:)=\prod_{i=0}^{k} \big(a_0+ia_1\big)^{c_{i}}.
\end{equation*}

\begin{lemma}\label{abc}
	The correspondences 
	\begin{enumerate}[label=(\alph*)]
		\item \label{aaba} $C\to A$, $\underline{c}\mapsto \sigma$ (with $\sigma_i=i+\sum_{j=0}^{i-1}c_j$) and
		\item \label{bbbcb} $B\to C$,  $\tilde{\sigma}\mapsto \underline{c}$ (with $c_i:=\# \{j: \tilde{\sigma}_j-j=i\}$)
	\end{enumerate}
	are well-defined bijections. Moreover $ (B,\beta) $ and $ (C,\delta)$ are isomorphic as weighted sets.
\end{lemma}

\begin{proof}
	For \ref{aaba}, well-definedness and injectivity are clear. Since the sets are finite and have the same cardinality\footnote{We notice it earlier $|A|=\binom{n}{k}$  and the known fact about weak composition of $n$ in $k+1$ parts.}, injectivity implies bijectivity.
	
	For \ref{bbbcb}, the map is well defined because each $c_i$ lies in $\{0,1,\coss,n\}$ and the sum $\sum c_i$ is precisely the length of $\tilde{\sigma},$ $|\tilde{\sigma}|=n-k$. Its inverse is defined recursively as follows:
	\begin{itemize}
		\item for $i=0$: for all $l$ with $1\leq l\leq c_0$, set $\tilde{\sigma}_l=l$;
		\item for $i=1,\dots,k$: for all $l$ with
		\[
		\sum_{j=0}^{i-1}c_j< l\leq \sum_{j=0}^{i} c_j,
		\]
		set $\tilde{\sigma}_l=l+i$.
	\end{itemize}
	By construction, $\tilde{\sigma}_1\geq 1$ and $\tilde{\sigma}_l<\tilde{\sigma}_{l+1}\leq (n-k)+k=n$, so $\tilde{\sigma}\in B$. It is also clear that this is the inverse of our correspondence.
	
	To show that it is a morphism (preserves weights), it suffices to use the definition of $c_i$, which counts the indices $j$ such that $\tilde{\sigma}_j-j=i$. 
\end{proof}
\noindent This lemma yields one of the main results of this paper. We obtain the following theorem.
\begin{theorem}\label{theocontr}
	Let $\big(\suit{T}\big)$ be a sequence satisfying \eqref{recu} with $\coul{a} = a_1 k + a_0$, $\coul{b} = b_2n+b_1k+b_0$, and the usual initial conditions \eqref{condiusue}. 
	Then
	\begin{equation}
	\suit{T} = \sum_{\{c_0+c_1+\coss+c_{k}=n-k, 0\leq c_i\leq n\}}  a_0^{c_0}\prod_{i=1}^{k} \big(a_0+a_1i\big)^{c_{i}} \Big(b_0+ (b_1+b_2)i+b_2 (\sum_{j=0}^{j=i-1}c_j) \Big).
	\end{equation}
\end{theorem}
\begin{proof}
	By Theorem \ref{princitheo}, we have 
	\[  \suit{T}
	=\sum_{\tilde{\sigma} \in \mathcal{C}^\uparrow(n-k,n)}
	\prod_{i=1}^{n-k} \big( a_0+ a_1( \tilde{\sigma}_i-i) \big)
	\prod_{i=1}^{k} \big( b_0+ib_1+b_2 \sigma_i \big). \] We apply Lemma \ref{abc} and conclude the result.
\end{proof}

\begin{appli}\rm\begin{enumerate}
		\item Recall that an $m$-Stirling permutation is a permutation $\sigma= \sigma_1\sigma_2\cdots\sigma_{mn}$ of the multiset $1^m2^m\cdots n^m$ such that for all $i,j, k$, if $i<j<k$ and $\sigma_i=\sigma_k$ then $\sigma_j\geq \sigma_i$. In \cite{He2023mthOrderEulerian}, 	
		Tian-Xiao He defines the $m$th order Eulerian numbers of order 
		as the number $T^{(r)}(n,k)$  of $m$-Stirling permutations having exactly $k$ descents, and proves that these numbers satisfy the following recurrence relation:
		\begin{equation}
		T^{(r)}(n,k)=\big(rn-k+(1-r)\big)T^{(r)}(n-1,k-1) + (k+1)T^{(r)}(n-1,k),
		\label{2.4}
		\end{equation}
		with $B^{(r)}(n,0)=1\mbox{ pour } $n$\geq 1\text{  et }~ B^{(r)}(n,k)=0$ if $n\leq k$ ou $k<0$.

		We apply Theorem \ref{theocontr} to get
		\begin{equation} 
		\suit{T^{(r)}} = \sum_{\{c_0+c_1+\coss+c_{k}=n-k, 0\leq c_i\leq n\}}  \prod_{i=1}^{k} \big(1+i)^{c_{i}} \Big( (r-1)(i-1)+r \sum_{j=0}^{j=i-1}c_j \Big).
		\end{equation}
		\item Let 's consider the $r$-Whitney-eulerian numbers defined by Mez\"o and Ramirez in \cite{Mezoram} (see also Toufik Mansour and all in \cite{manramshavil}) and satisfying the recurrence relation:
		\begin{equation}
		A_{m,r} (n, k) = (mk+r)A_{m,r}(n-1, k)+ (mn-mk+m-r) A_{m,r}(n-1, k-1).
		\end{equation}
		The $r$-Whitney-Eulerian numbers verify:
		\begin{equation}
		A_{m,r} (n, k) =\sum_{\{c_0+c_1+\coss+c_{k}=n-k, 0\leq c_i\leq n\}}  r^{c_0}\prod_{i=1}^{k} \big(r+mi)^{c_{i}} \Big(m-r+ m (\sum_{j=0}^{j=i-1}c_j) \Big).
		\end{equation}
	\end{enumerate} 
\end{appli}
\noindent 
By a similar way, one can obtain the general case, as follows.
\begin{theorem}\label{theocontr2}
	Let $\big(\suit{T}\big)$ be a sequence satisfying \eqref{recu} and the usual initial conditions \eqref{condiusue}. 
	Then
	\begin{equation}
		\suit{T}
		= \sum_{\{1\leq p_1<p_2<\cdots <p_{n-k}\leq n\}}
		\prod_{i=1}^{n-k} \Big( (a_2+a_1)p_i-a_1 i+a_0 \Big)
		\prod_{i=1}^{k} \Big(b_0+ (b_1+b_2)i+b_2 \sum_{j=0}^{i-1}\# \{l: p_l-l=j\} \Big).
	\end{equation}
\end{theorem}

\begin{porofo}
	First, the bijection that identifies $\tilde{\sigma}=(\tilde{\sigma}_i)_{1\leq i \leq n-k} \in \mathcal{C}^\uparrow(n-k,n)$ with $\sigma=(\sigma(i))_{1\leq i\leq k}\in\mathcal{C}^\uparrow(k,n)$ can be described as follows:
	\[
	\sigma(i)= i+\sum_{j=0}^{i-1} \# \{l: \tilde{\sigma}_l-l=j\}.
	\]
	Now, we can re-write Theorem~\ref{princitheo} for $a_{n,k}=a_2n+a_1k+a_0$ and $b_{n,k}=b_2n+b_1k+b_0$, as 
	\[
	\suit{T}
	=\sum_{\tilde{\sigma} \in \mathcal{C}^\uparrow(n-k,n)}
	\prod_{i=1}^{n-k} \big( a_2 \tilde{\sigma}_i+a_1(\tilde{\sigma}_i-i)+a_0 \big)
	\prod_{i=1}^{k} \Big( b_2\big(i+\sum_{j=0}^{i-1} \# \{l: \tilde{\sigma}_l-l=j\}\big)+b_1 i+b_0 \Big).
	\]
\end{porofo}

\noindent Now,  let's consider the generating series $T_n(x)=\sum_{k=0}^nT(n,k)x^k$ and $\overline{T}_k(y)=\sum_{n\geq k}T(n,k)y^n.$
\begin{theorem}
	The generating series $T_n(y)$ and $\overline{T}_k(y)$ satisfy:
	\begin{equation}
	T_n(x)=\big((b_2n+b_1+b_0)x+a_2n+a_0\big)T_{n-1}(x)+(b_1x^2+a_1x)T_{n-1}'(x)
	\end{equation} 
	and 
	\begin{equation}
	(1-(a_2+a_1k+a_0)y)\overline{T}_k(y)=a_2y^2\overline{T}_k'(y)+b_2y^2\overline{T}_{k-1}'(y)+(b_2+b_1k+b_0)y\overline{T}_{k-1}(y).
	\end{equation}
\end{theorem}
\begin{porofo}
	 We set  $\displaystyle\mathcal{R}_{n,*}:=\bigcup_{k\geq 0}\mathcal{R}_{n,k}$, $\displaystyle\mathcal{R}_{*,k}:=\bigcup_{n\geq k}\mathcal{R}_{n,k}$ and if $\pi\in\mathcal{R}_{n,k}$ then $l(\pi)=n$ and $h(\pi)=k$.
	We have
	 \begin{align*} T_n(x)&\displaystyle=\sum_{\pi\in\mathcal{R}_{n,*}}\omega_{a,b}(\pi)x^{h(\pi)}=\sum_{\pi\in\mathcal{R}_{n-1,*}}(b_2n+b_1(h(\pi)+1)+b_0)\omega_{a,b}(\pi)x^{h(\pi)+1}\\
		&\hspace{1cm} \displaystyle+\sum_{\pi\in\mathcal{R}_{n-1,*}}(a_2n+a_1h(\pi)+a_0)\omega_{a,b}(\pi)x^{h(\pi)}\\
		&\displaystyle=(b_2n+b_1+b_0)x\sum_{\pi\in\mathcal{R}_{n-1,*}}\omega_{a,b}(\pi)x^{h(\pi)}+b_1x^2\sum_{\pi\in\mathcal{R}_{n-1,*}}h(\pi)\omega_{a,b}(\pi)x^{h(\pi)-1}\\
		&\hspace{1cm} \displaystyle+(a_2n+a_0)\sum_{\pi\in\mathcal{R}_{n-1,*}}\omega_{a,b}(\pi)x^{h(\pi)}  +a_1x\sum_{\pi\in\mathcal{R}_{n-1,*}}h(\pi)\omega_{a,b}(\pi)x^{h(\pi)-1}\\
		&=\big((b_2n+b_1+b_0)x+a_2n+a_0\big)T_{n-1}(x)+(b_1x^2+a_1x)T_{n-1}'(x).
	\end{align*}
\noindent	In the other hand 
	\begin{align*}
		\overline{T}_k(y)&\displaystyle=\sum_{n\geq k}T(n,k)y^n= \sum_{\pi\in\mathcal{R}_{*,k}}\omega_{a,b}(\pi)y^{l(\pi)}\\&\displaystyle=\sum_{\pi\in\mathcal{R}_{*,k}}(a_2(l(\pi)+1)+a_1k+a_0)\omega_{a,b}(\pi)y^{l(\pi)+1}\\
		&\hspace{4cm} \displaystyle+\sum_{\pi\in\mathcal{R}_{*,k-1}}(b_2(l(\pi)+1)+b_1k+b_0)\omega_{a,b}(\pi)y^{l(\pi)+1}\\
		&\displaystyle=a_2y^2\sum_{\pi\in\mathcal{R}_{*,k}}l(\pi)\omega_{a,b}(\pi)y^{l(\pi)-1}+(a_2+a_1k+a_0)y\sum_{\pi\in\mathcal{R}_{*,k}}\omega_{a,b}(\pi)y^{l(\pi)}\\
		& \hspace{1cm} \displaystyle+b_2y^2\sum_{\pi\in\mathcal{R}_{*,k-1}}l(\pi)\omega_{a,b}(\pi)y^{l(\pi)-1}+(b_2+b_1k+b_0)y\sum_{\pi\in\mathcal{R}_{*,k-1}}\omega_{a,b}(\pi)y^{l(\pi)}\\
		&\displaystyle=a_2y^2\overline{T}_k'(y)+(a_2+a_1k+a_0)y\overline{T}_k(y)+b_2y^2\overline{T}_{k-1}'(y)+(b_2+b_1k+b_0)y\overline{T}_{k-1}(y).\\
	\end{align*}
\end{porofo}

\noindent Noticing $T_0(x)=1$ and if $a_2=0$, $\overline{T}_0(y)=\frac{1}{1-a_0y}$, we get directly:
\begin{corollary}
	If $b_1=a_1=0$, then 
	\begin{equation}
	T_n(x)=\prod_{i=1}^n\big((b_2i+b_1+b_0)x+a_2i+a_0\big).
	\end{equation}
	And if $a_2=b_2=0$ then 
	\begin{equation}
\overline{T}_k(y)= \frac{1}{1-a_0y}\prod_{i=1}^k\frac{(b_1k+b_0)y}{(1-(a_1k+a_0)y)}
	\end{equation}
\end{corollary}
\noindent These identities include the following results:
\begin{enumerate}
	\item The $r$-Whitney numbers of first kind $w_{m,r}(n,k)$, corresponding to $a_2=-m$, $a_1=0$, $a_0=m-r$, $b_2=b_1=0$ and $b_0=1$,  satisfy:
	\begin{equation}\label{keystirlinnumber111}
T_n(x)=\sum_{k=0}^nw_{m,r}(n,k)x^k=\prod_{i=0}^{n-1}(x-r-mi)
	\end{equation}
	\item The $r$-Whitney numbers of second kind $W_{m,r}(n,k)$, corresponding to $a_2=0$, $a_1=m$, $a_0=r$, $b_2=b_1=0$ and $b_0=1$,  satisfy:
	\begin{equation} \label{keystirlinnumber112}
\overline{T}_k(y)=\sum_{n\geq 0}W_{m,r}(n,k)y^k=\prod_{i=0}^k\frac{y}{(1-(mk+r)y)}.
	\end{equation}
\end{enumerate}
These results are seen in several combinatorial pieces of literature (\cite{rabeza}, \cite{Cheon2012rWhitney}).
\section{The sequence $(F(n,k))$ as transition matrix on the vector space $\mathbb{K}[x]$}\label{sect3}

Moreover, this approach permits us also to write some triangular arrays as transition matrix in the vector space $\mathbb{K}[x]$.
\begin{theorem}\label{theopass}
	\label{key}
	For any sequences $\big((F(n,k))\big)$, that satisfies the usual initial conditions
	\begin{equation*}
	F(0,0)=1, \quad \suit{F}=0 \text{ if } n<\max(k,0),
	\end{equation*} the following facts are equivalent
	\begin{equation}\label{bbbcas}
	\forall n,k: \quad		F(n,k) = F(n-1,k-1) + (a_2n+a_1k+a_0) F(n-1,k).
	\end{equation}
	and
	\begin{equation}\label{equabase1}
	\forall x: \quad 	\argtgen{x}{a_2}{n}= \sum_{k=0}^n \suit{F} \argtgenm{x-a_0-a_2}{a_1}{k}.
	\end{equation}
	where $ \argtgen{x}{a_2}{n}=\prod_{i=0}^{i=n-1}(x+ia_2)$ and $\argtgenm{x-a_0-a_2}{a_1}{k}=\prod_{i=0}^{k-1}(x-a_0-a_2-ia_1).$
\end{theorem}
\begin{porofo}
	Let $\st{R}{n}=\cup_{k\geq 0}\mathcal{R}_{n,k}$ be the set of paths of length $n$ in $\mathbb{N}\times \mathbb{N}$,  $\pi \in \st{R}{n}$. We define a  weighting on $\st{R}{n}$ by setting  $v(\pi)=\omega_{n,h(\pi)}(\pi)\argtgenm{x-a_0-a_2}{a_1}{h(\pi)}$ if where $\omega_{n,k}$ is the weighting associated with $F(n,k)$ and $h(\pi)$ is the height of $\pi$.  Then:
	\begin{equation}
	|\st{R}{n}|_v=\sum_{\pi\in \st{R}{n}}v(\pi)= \sum_{k\geq 0} \sum_{\pi\in \st{R}{n,k}}\omega_{n,k}(\pi)\argtgenm{x-a_0-a_2}{a_1}{k}=\sum_{k\geq 0} \suit{F} \argtgenm{x-a_0-a_2}{a_1}{k}. \notag
	\end{equation}
	Since a path $\pi$ of length $n$ can be obtained from a path $u$ of length $n-1$ by adding either an North–East step or a East step, we have
	
	\noindent	$\begin{array}{ll}|\mathcal{R}_n|_v&=\displaystyle\sum_{\pi\in\mathcal{R}_n}v(\pi)=\sum_{\pi\in\mathcal{R}_n}\omega_{n,h(\pi)}(\pi)(x-a_0-a_2|a_1)^{(\underline{h(\pi)})}\\ &\displaystyle =\sum_{u\in\mathcal{R}_{n-1}}(a_2n+a_1h(u)+a_0)\omega_{n-1,h(u)}(u)(x-a_0-a_2|a_1)^{(\underline{h(u)})}\\& \hspace{2cm} \displaystyle+\sum_{u\in\mathcal{R}_{n-1}}\omega_{n-1,h(u)}(u)(x-a_0-a_2|a_1)^{(\underline{h(u)+1})}\\
	& =\displaystyle	\sum_{u\in\mathcal{R}_{n-1}}[(a_2n+a_1h(u)+a_0)+(x-a_0-a_2-h(u)a_1)]\omega_{n-1,h(u)}(u)(x-a_0-a_2|a_1)^{(\underline{h(u)})}\\
	&=(x+a_2(n-1))|\mathcal{R}_{n-1}|_v.
	\end{array}$

	Noting that $|\st{R}{0}|_w=1$, we obtain the result.

\end{porofo}
\noindent
Equation (\ref{equabase1}) generalizes many results. For examples:
\begin{appli}\rm
	\begin{enumerate}	
		\item The generalized Stirling Numbers $S_{\alpha,\beta,r}$ verify: 
		\begin{equation}
		(x|-\alpha)^{(\overline{n})}=\sum_{k=0}^nS_{\alpha,\beta,r}(n,k)(x-r|\beta)^{(\underline{k})}.
		\end{equation} 
		We obtain Equation (1) of \cite{hsu1998unified} (seen also in \cite{Maier2023}).
		In particular, the $r$-Whitney numbers of first kind, verify:  \begin{equation}\label{equamezo1}\prod_{i=0}^{n-1}(x-mi)=\sum_{k=0}^nw_{m,r}(n,k)(x+r)^k.\end{equation}
		By replacing $x$ with $mx$ in this equation, we obtain Equation (2) of \cite{Mezo2010Bernoulli}.		
		And the $r$-Whitney number of second kind, verify: \begin{equation}\label{equamezo2}x^n=\sum_{k=0}^nW_{m,r}(n,k)\prod_{i=0}^{n-k}(x-r-im).\end{equation}
		By replacing $x$ with $mx+r$ in this equation, we obtain Equation (1) of \cite{Mezo2010Bernoulli}.

		\item If we consider the sequence $(\lambda_{m,r}(n,k)$ defined in application \ref{appli1}:
		we have:
		
		\begin{equation}\label{equabase2}
		\argtgen{x}{m}{n}=  \sum_{k=0}^n \lambda_{m,r}(n,k \argtgenm{x-r}{1}{k}.
		\end{equation}
		
		\item	 If we consider the sequence $(\Lambda_{m,r,s}(n,k)$,
		we have: \begin{equation}
		(x|m)^{(\overline{n})}=\sum_{k=0}^n\Lambda_{m,r,s}(n,k)(x-r-s-m|m)^{(\underline{k})}.
		\end{equation} 
		
		\item If we consider the sequence $\tau_{m,r}(n,k)$, with $a_1=a_2=m$ and $a_0=mr$, we obtain: 
		
		\begin{equation}\label{equabase3}
		\argtgen{x}{m}{n}= \sum_{k=0}^n \suit{\tau_{m,r}} \argtgenm{x-mr-m}{m}{k},
		\end{equation}
	\end{enumerate}
	
\end{appli}
\begin{remark} 	\rm
	We  note that $\big(\argtgen{x}{a_2}{n}\big)_{n\idc}$ form a basis of $\fiel{K}[x]$. The same holds for $\big(\argtgenm{x-a_0-a_2}{a_1}{k}\big)_{k\idc}$. Therefore, Theorem \ref{theopass} says that $\mathbb{F}^T=\big(\suit{F}\big)_{k,n\in \natu}$ is a transition matrix from$\big(\argtgenm{x-a_0-a_2}{a_1}{k}\big)_{k\idc}$  to $\big(\argtgen{x}{a_2}{n}\big)_{n\idc}$. The matrix $\mathbb{F}^T$
	is therefore invertible.
\end{remark}
\begin{theorem}\label{inverse}
	The inverse of  $\mathbb{F}^T$, is the matrix $\mathbb{H}^T=\big((H(n,k))\big)$ satisfying:
	\begin{equation}\label{reps}
	\forall n,k: \quad		H(n,k) = H(n-1,k-1) + (-a_1n-a_2k+a_1-a_0-a_2) H(n-1,k).
	\end{equation}
\end{theorem}
\begin{porofo}
	Consider the inverse $\big(H(n,k)\big)$ of $\big(F(n,k)\big)$. Using \eqref{equabase1}, we have
	$$\displaystyle
	\forall x: \quad \argtgenm{x-a_0-a_2}{a_1}{k}	= \sum_{k=0}^n \suit{H}  \argtgen{x}{a_2}{n}.$$
	We can change $y=x-a_0-a_2$ and use the fact $\argtgen{z}{c}{l}=\argtgenm{z}{-c}{l}$ in this equation to get 
	\[ \forall y: \quad 	\argtgen{y}{-a_1}{n}= \sum_{k=0}^n \suit{H} \argtgenm{y+a_0+a_2}{-a_2}{k}.\]
	Set $a'_2:=-a_1;\quad a'_1:=-a_2; \quad a'_0:=a_1-a_0-a_2,$ and plug in this equation. Then 
	\[\displaystyle
	\forall y: \quad 	\argtgen{y}{a'_2}{n}= \sum_{k=0}^n \suit{H} \argtgenm{y-a'_0-a'_2}{a'_1}{k}.\]
	So applying Theorem \ref{key},  we get
	$$\forall n,k: \quad		H(n,k) = H(n-1,k-1) + (a'_2n+a'_1k+a'_0) H(n-1,k).$$
\end{porofo}

\section{Closed formula for $F(n,k)$, and for $T(n,k)$ if $b_{n,k}=b_1k+b_0$}\label{sect4}
We now propose to prove a closed formula for the sequences $F(n,k)$.  The case of the sequences $\big(T(n,k)\big)$, with $b_{n,k}=b_0+b_1k$ then follows as a direct consequence.
\subsection{The case where $a_1\neq 0$}
\noindent For this, we need the following two lemmas.
\begin{lemma}\label{lem11}
	Let $P_n(x):=\prod_{l=1}^{l=n}(a_0+a_1x+a_2l)$. If $a_1\neq 0$, we have 
	\[F(n,k)=\frac{1}{a_1^k}[x^{(\underline{k})}]P_n(x).\]
\end{lemma}
\begin{porofo} Note that $P_n(x)=(a_0+a_1x+a_2n)P_{n-1}(x)$ and $x\cdot x^{(\underline{k})}=x^{(\underline{k+1})}+k \cdot x^{(\underline{k})}$. These facts imply that $G(n,k):=[x^{(\underline{k})}]P_n(x)$ satisfies 
	\[ G(0,0)=1 ; \quad G(n,k)= (a_0+a_1k+a_2n)G(n-1,k)+a_1G(n-1,k-1). \]
	Using Lemma \ref{cor1}, we have $G(n,k)=a_1^k F(n,k)$.
\end{porofo}
\begin{lemma}\label{lem22} For any polynomial $P(x)=\sum_{i=0}^{i=n}\alpha_i x^i$, 
	\[[x^{(\underline{k})}]P(x)=\frac{1}{ k!} \sum_{j=0}^{k} (-1)^{\,k-j} \binom{k}{j} P(j).\]
\end{lemma}
\begin{porofo}
	We have $x^i=\sum_{k\geq 0} S(i,k)x^{(\underline{k})}$ (see Comtet \cite{comtet}). Then
	
	$\begin{array}{ll}
	P(x)=&\sum_{i=0}^{n}\alpha_i\sum_{k= 0}^{n} S(i,k)x^{(\underline{k})}
	=\sum_{k= 0}^{n} \Big( \sum_{i=0}^{n}\alpha_i S(i,k) \Big) x^{(\underline{k})}\\
	&=\sum_{k= 0}^{n} \Big( \sum_{i=0}^{n}\alpha_i \frac{1}{k!} \sum_{j=0}^{j=k} (-1)^{k-j}\binom{k}{j} j^i \Big) x^{(\underline{k})}\\
	&=\sum_{k= 0}^{n} \Big(  \frac{1}{k!} \sum_{j=0}^{k} (-1)^{k-j}\binom{k}{j}  ( \sum_{i=0}^{n}\alpha_i j^i) \Big) x^{(\underline{k})}.\\
	\end{array}$
	
	So the conclusion.
\end{porofo}
\noindent
Applying Lemma \ref{lem22} to the polynomial $P_n(x)$ of Lemma \ref{lem11}, we obtain:
\begin{theorem}\label{ffromgram}
	Let $\big(\suit{F}\big)$ be a sequence satisfying \eqref{recu} with $\coul{a} = a_2 n + a_1 k + a_0$, $\coul{b} = 1$, and the usual initial conditions \eqref{condiusue}. 
	
	If $a_1 \neq 0$, then
	\begin{equation}\label{fnkfromgram}
	\suit{F} = \frac{1}{a_1^k\, k!} \sum_{j=0}^{k} (-1)^{\,k-j} \binom{k}{j} \prod_{l=1}^{n} (a_0 + a_1 j + l a_2).
	\end{equation}
\end{theorem}
\begin{appli}\rm 
	\begin{enumerate}
		\item	 The generalized Stirling numbers verify:
		\begin{equation}
		S_{\alpha,\beta,r}(n,k)=\frac{1}{\beta^k\, k!} \sum_{j=0}^{k} (-1)^{\,k-j} \binom{k}{j} \prod_{l=0}^{n-1} (r + \beta j - l \alpha)
		\end{equation}
		
		In particular,
		the $r$-Whitney numbers $W_{m,r}(n,k)$ of second kind satisfy: 
		\begin{equation}
		W_{m,r}(n,k)= \frac{1}{m^k\, k!} \sum_{j=0}^{k} (-1)^{\,k-j} \binom{k}{j} (r + m j)^n.
		\end{equation}
		
		In this case we obtain a result of Mez\"o and Ramirez in \cite{Mezoram} (Corollary 8).
		\item The numbers $\lambda_{m,r}(n,k)$ of Equation (\ref{equav3}) verify:
		\begin{equation}
		\suit{\lambda_{m,r}} = \frac{1}{k!} \sum_{j=0}^{k} (-1)^{\,k-j} \binom{k}{j} \prod_{l=1}^{n} (r + j + ml).
		\end{equation}
		\item The numbers $\Lambda_{m,r,s}(n,k)$of Equation (\ref{equak}) satisfy:
		\begin{equation}
		\Lambda_{m,r,s}(n,k)= \frac{1}{m^k\, k!} \sum_{j=0}^{k} (-1)^{\,k-j} \binom{k}{j} \prod_{l=j}^{n+j} (r+s + ml).
		\end{equation}
		In particular, the $r$-Whitney-Lah numbers $\Lambda_{m,r,r}$
		verify:\begin{equation}
		\Lambda_{m,r,r}(n,k)= \frac{1}{m^k\, k!} \sum_{j=0}^{k} (-1)^{\,k-j} \binom{k}{j} \prod_{l=j}^{n+j} (2r + ml).
		\end{equation}
		And the $r$-Lah numbers $\Lambda_{1,r,r}(n,k)$
		verify:
		\begin{equation}
		\Lambda_{1,r,r}(n,k)= \frac{1}{k!} \sum_{j=0}^{k} (-1)^{\,k-j} \binom{k}{j} \prod_{l=j}^{n+j} (2r + l).
		\end{equation}
		
		\item The numbers $\tau_{m,r}(n,k)$ of Equation (\ref{equatau}) verify:
		\begin{equation}
		\suit{\tau_{m,r}} = \frac{m^{n-k}}{k!} \sum_{j=0}^{k} (-1)^{\,k-j} \binom{k}{j} \prod_{l=j+1}^{j+n} (r + l)=.\frac{m^{n-k}}{k!} \sum_{j=0}^{k} (-1)^{\,k-j} \binom{k}{j} \frac{(r+j+n)!}{(r+j)!}
		\end{equation}
	\end{enumerate}
\end{appli}

\noindent Let's consider the generating series $F_n(x)=\sum_{k=0}^nF(n,k)x^k$ and $\overline{F}_k(y)=\sum_{n\geq k}F(n,k)y^k.$
Then $$\displaystyle F_n(x)=\sum_{\pi\in\mathcal{R}_{n,*}}\omega_{a,b}(\pi)x^{h(\pi)}=\sum_{\pi\in\mathcal{R}_{n-1,*}}\omega_{a,b}(\pi)x^{h(\pi)+1}+\sum_{\pi\in\mathcal{R}_{n-1,*}}(a_2n+a_1h(\pi)+a_0)\omega_{a,b}(\pi)x^{h(\pi)}.$$ 
Since $\sum_{\pi\in\mathcal{R}_{n-1,*}}\omega_{a,b}(\pi)h(\pi)x^{h(\pi)}=x\sum_{\pi\in\mathcal{R}_{n-1,*}}h(\pi)\omega_{a,b}(\pi)x^{h(\pi)-1}$, it follows that
\begin{equation}
F_n(x)=(x+a_2n+a_0)F_{n-1}(x)+a_1xF'_{n-1}(x).
\end{equation}
In the other hand, setting $\displaystyle\mathcal{R}_{*,k}=\bigcup_{n\geq k}\mathcal{R}_{n,k}$, we have
\begin{align*}
\overline{F}_k(y)&=\sum_{n\geq k}F(n,k)y^n= \sum_{\pi\in\mathcal{R}_{*,k}}\omega_{a,b}(\pi)y^{l(\pi)}\\&=\sum_{\pi\in\mathcal{R}_{*,k}}(a_2(l(\pi)+1)+a_1k+a_0)\omega_{a,b}(\pi)y^{l(\pi)+1}+\sum_{\pi\in\mathcal{R}_{*,k-1}}\omega_{a,b}(\pi)y^{l(\pi)+1}\\
&=a_2y^2\sum_{\pi\in\mathcal{R}_{*,k}}l(\pi)\omega_{a,b}(\pi)y^{l(\pi)-1}+(a_2+a_1k+a_0)y\sum_{\pi\in\mathcal{R}_{*,k}}\omega_{a,b}(\pi)y^{l(\pi)}\\
& \hspace{1cm} +y\sum_{\pi\in\mathcal{R}_{*,k-1}}\omega_{a,b}(\pi)y^{l(\pi)}\\
&=a_2y^2\overline{F}_k'(y)+(a_2+a_1k+a_0)y\overline{F}_k(y)+y\overline{F}_{k-1}(y).\\
\end{align*}
That is
\begin{equation}
(1-(a_0+a_1k+a_0))\overline{F}_k(y)=y\overline{F}_{k-1}(y)+a_2y^2\overline{F}_k'(y).
\end{equation}

\subsection{The case where $a_1= 0$}
Let us now consider the case $a_1=0$ and $a_2\neq 0$. Define $Q_n(x)=\prod_{i=1}^{n}(x+a_0+a_2i)$.
On the one hand,  we have: $Q_n(x)=\displaystyle\sum_{k=0}^n\big(\sum_{ \displaystyle\substack{A\subset [n], |A|=n-k}}\prod_{i\in A}(a_0+ia_2)\big)x^k$.
On the other hand, we have $\prod_{i=0}^{n-1}(u+i)=\sum_{j=0}^nc(n,j)u^j$, where $c(n,j$ is the unsigned Stirling numbers of first kind (see Comptet \cite{comtet}). Then:

\noindent	 $\begin{array}{ll} Q_n(x)&=\displaystyle\prod_{i=0}^{n-1}(x+(a_0+a_2)+a_2i)
=a_2^n\prod_{i=0}^{n-1}(\frac{x+a_0+a_2}{a_2}+i)\displaystyle=a_2^n\sum_{j=0}^{n}c(n,j)(\frac{x+a_0+a_2}{a_2})^j\\&
\displaystyle=a_2^n\sum_{j=0}^{n}c(n,j)(\sum_{k=0}^j{j\choose k}a_2^{-j}x^k(a_0+a_2)^{j-k}) \displaystyle=\sum_{k=0}^{n}(\sum_{j=k}^nc(n,j){j\choose k}a_2^{n-j}(a_0+a_2)^{j-k})x^k.
\end{array}$

\noindent 
So 
\begin{equation}[x^k]Q_n(x)=\displaystyle\sum_{ \substack{A\subset [n], |A|=n-k}}\prod_{i\in A}(a_0+ia_2)=\displaystyle\sum_{j=k}^nc(n,j){j\choose k}a_2^{n-j}(a_0+a_2)^{j-k}.
\end{equation}
Let $\big(\suit{F}\big)$ be a sequence satisfying \eqref{recu} with $\coul{a} = a_2 n + a_0$, $\coul{b} = 1$, and the usual initial conditions \eqref{condiusue}. 
The weight of each East step \es\ $(i-1,j)\to (i,j)$ by $\cou{a}{i}{j}$ is independent of $j$ and is equal to $a_0+a_2i$ when it is  the $i$-th step. Thus, if $\pi\in\mathcal{R}_{n,k}$, and if $A_\pi$  is the set of positions of the East steps of $\pi$, then its weight is $\omega(\pi)=\prod_{i\in A_\pi}(a_0+a_2i)$. Since each path $\pi\in\mathcal{R}_{n,k}$ can be identified with the set $A_{\pi}$
and $|A_\pi|=n-k$, it follows that 
$$F(n,k)=\displaystyle\sum_{\pi\in\mathcal{R}_{n,k}}\omega(\pi)=\displaystyle\sum_{\substack{A\subset [n], |A|=n-j}}\prod_{i\in A}(a_0+ia_2).$$ Hence
we obtain
\begin{theorem}\label{a1zero} If $\big(\suit{F}\big)_{n,k\in \mathbb{N}}$ be a sequence satisfying \eqref{recu} with $\coul{a} = a_2 n + a_0$, $\coul{b} = 1$, and the  usual initial conditions \eqref{condiusue}, then:
	\begin{equation}
	\sum_{k=0}^nF(n,k)x^k=\prod_{i=1}^{n}(x+a_0+a_2i).
	\end{equation}
	That is:
	\begin{equation}
	F(n,k)=\sum_{j=k}^nc(n,j){j\choose k}a_2^{n-j}(a_0+a_2)^{j-k}.
	\end{equation} 
	
\end{theorem}
\begin{appli}\rm 
	\begin{enumerate}
		\item 
		Let us consider the $r$-Whitney numbers of first kind obtained by setting 
		$a_2=-m$, $a_0=m-r$. Then 
		\begin{equation}
		w_{m,r}(n,k)=\sum_{j=k}^n(-1)^{n-k}m^{n-j}c(n,j){j\choose k}r^{j-k}
		\end{equation} 
		We obtain a result of Randrianirina in \cite{rabeza} (Equation (37)).
		\item 	More generally, let $\big(F(n,k)\big)$ be a sequence satisfying:
		\begin{equation}
		F(n,k)=F(n-1,k-1)+ (a_0+a_1k)F(n-1,k).
		\end{equation}
		Equation \ref{fnkfromgram} say that:
		\begin{equation}\label{fnkfromgram12}
		\suit{F} = \frac{1}{a_1^k\, k!} \sum_{j=0}^{k} (-1)^{\,k-j} \binom{k}{j} (a_0 + a_1 j)^n.
		\end{equation}
		Theorem  \ref{inverse} says that the inverse of the matrix $\mathbb{F}^T=(F(n,k))_{n,k}$, is the matrix $\mathbb{H}^T=((H(n,k)))$ satisfying
		\begin{equation}
		H(n,k) = H(n-1,k-1) + (-a_1n+a_1-a_0) H(n-1,k).
		\end{equation}
		We have 
		\begin{equation} H(n,k)=\sum_{j=k}^nc(n,j){j\choose k}(-1)^{n-k}a_1^{n-j}a_0^{j-k}.
		\end{equation}
	\end{enumerate}
\end{appli}
\begin{corollary}\label{tonerw2}
	Let $\big(\suit{F}\big)$ be a sequence satisfying \eqref{recu} with $\coul{a} = a_2 n + a_1 k + a_0$, $\coul{b} = 1$, and the usual initial conditions \eqref{condiusue}. We have
	\begin{equation}
F(n,k)= \sum_{j=k}^{n}
\left(
\sum_{i=j}^{n}
c(n,i)\binom{i}{j}
a_2^{n-i}
(a_2+a_0)^{i-j}
\right)
a_1^{j-k}
S(j,k).
	\end{equation}
That is 
\begin{equation}
	F(n,k)=
	\sum_{i=k}^{n}\sum_{j=k}^{i}
	\binom{i}{j} c(n,i) S(j,k)
	a_2^{n-i}
	(a_2+a_0)^{i-j}
	a_1^{j-k}.
\end{equation}
\end{corollary}

\begin{proof}
	Let $\big(U(n,k)\big)$ and $\big(V(n,k)\big)$ be two triangular sequences satisfying the usual initial conditions \eqref{condiusue}, with
	\begin{equation*}
		U(n,k) = (a_2 n + a_0) U(n-1,k) + U(n-1,k-1),
	\end{equation*}  
	and
	\begin{equation*}
		V(n,k) = a_1 k V(n-1,k) + V(n-1,k-1).
	\end{equation*}
	It is straightforward to see that $\mathbb{F} = \mathbb{U} \mathbb{V}$. On the other hand, the weights corresponding to $\big(V(n,k)\big)$ differ from those of $S(n,k)$ only by a factor of $a_1^{\,n-k}$, since $S(n,k)$ satisfies $S(n,k) = k S(n-1,k) + S(n-1,k-1)$. Hence, we have $V(n,k) = a_1^{\,n-k} S(n,k)$. Applying Theorem \ref{a1zero}, we obtain the desired conclusion.
\end{proof}

\subsection{The case where $\coul{a} = a_2 n + a_1 k + a_0$ and $\coul{b} = b_1k+b_0$}
Consider the sequence $(T(n,k))$ satisfying:
\begin{equation}
T(n,k)=(a_2n+a_1k+a_0)T(n-1,k)+(b_1k+b_0)T(n-1,k-1)
\end{equation}
with initial condition $T(0,0)=1$ and $T(n,k)=0$ if $n>k$.

We combine Lemma \ref{cor1} with Theorems \ref{ffromgram} and \ref{a1zero} to obtain the following.
\begin{theorem}
	Let $\big(\suit{T}\big)$ be a sequence satisfying \eqref{recu} with $\coul{a} = a_2 n + a_1 k + a_0$, $\coul{b} = b_1k+b_0$, and the usual initial conditions \eqref{condiusue}. 
	
	If $a_1 \neq 0$, then
	\begin{equation}\label{tnkn}
	\suit{T} = \frac{\argtgen{b_0+b_1}{b_1}{k}}{a_1^k\, k!} \sum_{j=0}^{k} (-1)^{\,k-j} \binom{k}{j} \prod_{r=1}^{n} (a_0 + a_1 j + r a_2),
	\end{equation} where $\argtgen{b_0 + b_1}{\, b_1 \,}{k}=(b_0 + b_1) \times (b_0 + 2b_1) \times \coss \times (b_0 + k b_1). $
	
	If $a_1=0$, then 
	\begin{equation}\label{tnk2}
	T(n,k)=\argtgen{b_0 + b_1}{\, b_1 \,}{k}\sum_{j=k}^na_2^{n-j}c(n,j){j\choose k}a_0^{j-k}.
	\end{equation}
\end{theorem}
\noindent
One can check that equation~\eqref{tnk2} coincides with that of Neurwirth~\cite{neuwirth2001}. But in general we can get the results of Neurwirth using these theorems. This is stated in Spivey~\cite{Spivey2011JIS} (equation (2)) as follows, with minor corrections and slight adjustments.
\begin{theorem}
	Let $\big(\suit{T}\big)$ be a sequence satisfying \eqref{recu} with $\coul{a} = a_2 n + a_1 k + a_0$, $\coul{b} = b_1k+b_0$, and the usual initial conditions \eqref{condiusue}. Then
	\begin{equation}
		T(n,k)=\argtgen{b_0 + b_1}{\, b_1 \,}{k} 	\sum_{i=k}^{n}\sum_{j=k}^{i}
		\binom{i}{j} c(n,i) S(j,k)
		a_2^{n-i}
		(a_2+a_0)^{i-j}
		a_1^{j-k}.
	\end{equation}
\end{theorem}
\noindent This now follows directly from Lemma \ref{cor1} and Corollary \ref{tonerw2}.

\end{document}